\documentclass[12pt,a4paper,reqno]{amsart}
\usepackage{amsmath}
\usepackage{amsfonts}
\usepackage{amssymb,amsthm}
\numberwithin{equation}{section}

\def\A{{\mathcal A}}

\def\Z{{\mathbb Z}}

\def\pmod #1{\ ({\rm{mod}}\ #1)}

\theoremstyle{plain}
\newtheorem{theorem}{Theorem}

\newtheorem{corollary}{Corollary}

\theoremstyle{definition}
\newtheorem*{acknowledgment}{Acknowledgments}
\theoremstyle{remark}
\newtheorem{remark}{Remark}

\begin{document}

\title
{A note on the sum of reciprocals}

\author{Yuchen Ding}
\address {Department of Mathematics, Nanjing
University, Nanjing 210093, People's Republic of China}
\email{dg1721001@smail.nju.edu.cn}
\author{Yu-Chen Sun}
\address {Medical School, Nanjing
University, Nanjing 210093, People's Republic of China}
\email{syc@smail.nju.edu.cn}



 \begin{abstract}
 For a fixed positive integer $m$ and any partition $m = m_1 + m_2 + \cdots + m_e$ , there exists a sequence $\{n_{i}\}_{i=1}^{k}$ of positive integers such that
$$m=\frac{1}{n_{1}}+\frac{1}{n_{2}}+\cdots+\frac{1}{n_{k}},$$
with the property that partial sums of the series $\{\frac{1}{n_i}\}_{i=1}^{k}$ can only represent the integers with the form $\sum_{i\in I}m_i$, where $I\subset\{1,...,e\}$.

\end{abstract}

\maketitle

\section{Introduction}
\setcounter{lemma}{0}
\setcounter{theorem}{0}
\setcounter{corollary}{0}
\setcounter{remark}{0}
\setcounter{equation}{0}
\setcounter{conjecture}{0}

Let $a+n\Z$ denote the arithmetic progression $\{x\in\Z: x\equiv a \pmod{n}\}$. For A finite system $ \A=\{a_i+n_i \Z\}_{i=1}^{k} $, we define the covering function $w_{\A}$ over $\Z$ by
$$
w_{\A}:=|\{1 \leq i \leq k: x \in a_i+n_i \Z \}|
$$
the system $\A$ is called an $m$-cover of $\Z$, when $w_{\A}\geq m$ for all integers $x$. In particular, we say that $\A$ is an exact $m$-cover if $w_{\A}=m$ for all $x\in \Z$
which all integers belong to exactly $m$ times is called an exactly $m$-cover.

The conception of covering in $\Z$ was first mentioned by Erd\"os \cite{E50} and has been investigated in many papers (e.g.\cite{BFF88,C03,GS05,P74,S00,PZ09}).

 In \cite{P76}, Porubsk\'y firstly studied the exact $m$-cover and get an evident but interesting result
 $$\sum_{i=1}^{k}\frac{1}{n_i} = m.$$ Also for $m$-cover, clearly we have $\sum_{i=1}^{k}1/n_{i} \geq m$.
 
 In \cite{Z89}, Zhang discovered an attractive and not trivial connection between covering systems and Egyptian fractions. Provided $\A$ is a $1$-cover of $\Z$, He showed that
 $$\sum_{s\in I}\frac{1}{n_s} \in \Z^+,$$
 for some $I\subset\{1,...,k\}$.

 Furthermore, Sun did a lot of researches about how to combine the sum of reciprocals of residue classes modules with the covering system. For example, in $1992$, Sun has showed that for each $n=1,...,m$, there exist (at least) $\binom{k}{2}$ subsets $I$ of $\{1,...,k\}$ such that
  $$\sum_{s\in I}\frac{1}{n_i}=n.$$
 For more related results, readers may refer to \cite{S95,S96}.
 
 Suppose that $\A_1,\A_2$ are $m_1$-cover and $m_2$-cover over $\Z$, obviously $\A=A_1\cup\A_2$ constitutes a $(m_1+m_2)$-cover. Conversely, in \cite{P76} Porubsk\'y asked that whether for each $m\geq 2$ there exists an exact $m$-cover of $\Z$ which cannot be split into an exact $n$-cover and exact $(m-n)$ cover with $1\leq n<m$. Then Zhang \cite{Z91} claimed that the question mentioned by Porubsk\'y is true.

Motivated by above researches, we may consider whether for each $m\geq 1$ there exist a series $\{n_i\}_{i=1}^{k}$ satisfies that $m=\sum_{i}^{k}1/n_i$ and any partial sum of $\{n_i\}_{i=1}^k$ on longer belongs to $\Z$.

\begin{theorem}\label{main1}
For every given positive integer $m$, there exists a sequence $\{n_{i}\}_{i=1}^{k}$ of positive integers satisfying that
$$m = \frac{1}{n_{1}}+\frac{1}{n_{2}}+\cdots+\frac{1}{n_{k}}$$
with $\sum_{i\in I}\frac{1}{n_{i}}\not\in \mathbb{Z}$ for any $\emptyset \not = I\subsetneqq \{1,2,...,k\}$.
\end{theorem}

In fact, we can prove a stronger result. 

\begin{theorem}\label{main2}
Let $m$ be a given positive integer, for any partition $m=m_{1}+m_{2}+\cdots+m_{e},1\leq m_{1}\leq m_{2}\leq\cdots\leq m_{e}\leq m$, there exists a sequence $\{n_{k}\}$ of positive integers such that
$$m=\frac{1}{n_{1}}+\frac{1}{n_{2}}+\cdots+\frac{1}{n_{k}}$$
which satisfies
$$\Omega\{n_{1},...,n_{k}\}=\bigg\{m_{i_{1}}+
\cdots+m_{i_{d}}:1\leq i_{1}<\cdots< i_{d}\leq e,\{i_{1},...,i_{d}\}\subset\{1,...,e\}\bigg\},$$
where $\Omega\{n_{1},n_{2},...,n_{k}\}=\bigg\{\sum\limits_{i\in I}\frac{1}{n_{i}}:I\subset \{1,2,...,k\}\bigg\}\cap\mathbb{Z}$. One may also require that
$$\bigg|\bigg\{I:I\subset \{1,2,...,k\},\sum\limits_{i\in I}\frac{1}{n_{i}}\in\mathbb{Z}\bigg\}\bigg|=2^{e}.$$
\end{theorem}

Let us give a simple explanation about why Theorem \ref{main2} implies Theorem \ref{main1}. We choose $e=1$ and exclude the case that the subset $I$ of $\{1,...,k\}$ equal to $\emptyset$, Then we obtain the desired result.

\begin{corollary}\label{simple2}
For every given positive integer $m$, there exists a sequence $\{n_{i}\}_{i=1}^{k}$ of positive integers satisfying that
$$m = \frac{1}{n_{1}}+\frac{1}{n_{2}}+\cdots+\frac{1}{n_{k}}$$
with $\sum_{i\in I}\frac{1}{n_{i}}\in \{n, m-n\}$ for any positive integer $1\leq n< m$ and $\emptyset \not = I\subsetneqq \{1,2,...,k\}$.
\end{corollary}

\begin{remark}
This corollary is immediately evident by our Theorem with the case $e=2$. We think this is very related to Zhang's result \cite{Z91} which affirmed the Porubsk\'y question \cite{P76} that whether for each $m\geq 2$ there exists an exact $m$-cover of $\Z$ which cannot be split into an exact $n$-cover and exact $(m-n)$-cover with $1\leq n<m$ is true.
\end{remark}
\section{Proof of Theorem \ref{main2}}
\setcounter{lemma}{0}
\setcounter{theorem}{0}
\setcounter{corollary}{0}
\setcounter{remark}{0}
\setcounter{equation}{0}
\setcounter{conjecture}{0}

\begin{proof}[Proof of Theorem \ref{main2}]
We use $p_{j}$ to denote the $jth$ prime. It is well known that the reciprocals of all the primes diverges. So there exists $e$ integers $0=l_{1}<l_{2}<\cdots<l_{e}<l_{e+1}$ such that
$$\sum_{j=l_{c}+1}^{l_{c+1}}\frac{1}{p_{j}}<m_{c}<\sum_{j=l_{c}+1}^{l_{c+1}+1}\frac{1}{p_{j}}$$
for $c=1,2,...,e$.
For any $1\leq c\leq e$, suppose that
$$m_{c}=\sum_{j=l_{c}+1}^{l_{c+1}}\frac{1}{p_{j}}+\frac{g_{c}}{\prod_{j=l_{c}+1}^{l_{c+1}}p_{j}}.$$
It is clear that $g_{c}<\prod\limits_{j=l_{c}+1}^{l_{c+1}+1}p_{j}$ and $(g,\prod\limits_{j=l_{c}+1}^{l_{c+1}+1}p_{j})=1$, we can choose $s_{c},t_{c}\in \mathbb{Z^{+}}$ such that
$$g_{c}s_{c}-t_{c}\prod\limits_{j=l_{c}+1}^{l_{c+1}+1}p_{j}=1.$$
For any $a\in\mathbb{Z^{+}}$, the above equality remains when we replace $s_{c},t_{c}$ by
$\tilde{s_{c}}=s_{c}+a\prod\limits_{j=l_{c}+1}^{l_{c+1}+1}p_{j}$ and $\tilde{t_{c}}=t_{c}+ag_{c}$ separately. By the celebrated Dirichlet's theorem in arithmetic progressions, we may choose $\tilde{s_{c}}$ to be $e$ different primes with $p_{l_{e+1}}<\tilde{s_{0}}<\cdots<\tilde{s_{e}}$.
Noting that $\tilde{t_{c}}<\tilde{s_{c}}$, and we have
$$m_{c}=\sum_{j=l_{c}+1}^{l_{c+1}}\frac{1}{p_{j}}+\frac{\tilde{t_{c}}}{\tilde{s_{c}}}+\frac{1}{\tilde{s_{c}}\prod_{j=l_{c}+1}^{l_{c+1}}{p_{j}}}.$$
Therefore we get
$$m=\sum_{c=1}^{e}\bigg\{\sum_{j=l_{c}+1}^{l_{c+1}}\frac{1}{p_{j}}+\frac{\tilde{t_{c}}}{\tilde{s_{c}}}+
\frac{1}{\tilde{s_{c}}\prod_{j=l_{c}+1}^{l_{c+1}}{p_{j}}}\bigg\}:=\sum\limits_{i=1}^{k}\frac{1}{n_{i}}.$$
It remains to show that
$$\Omega\{n_{1},...,n_{k}\}=\bigg\{m_{i_{1}}+
\cdots+m_{i_{d}}:1\leq i_{1}<\cdots< i_{d}\leq e,\{i_{1},...,i_{d}\}\subset\{1,...,e\}\bigg\},$$
and
$$\bigg|\bigg\{I:I\subset \{1,2,...,k\},\sum\limits_{i\in I}\frac{1}{n_{i}}\in\mathbb{Z}\bigg\}\bigg|=2^{e}.$$
By the construction of $m$, we find that
$$\bigg\{m_{i_{1}}+\cdots+m_{i_{d}}:1\leq i_{1}<\cdots< i_{d}\leq e,\{i_{1},...,i_{d}\}\subset\{1,...,e\}\bigg\}
\subset\Omega\{n_{1},...,n_{k}\}$$
and
$$\bigg|\bigg\{(i_{1},...,i_{d}):1\leq i_{1}<\cdots< i_{d}\leq e,\{i_{1},...,i_{d}\}\subset\{1,...,e\}\bigg\}\bigg|=2^{e}.$$
It follows that we need only to show that
$$\bigg|\bigg\{I:I\subset \{1,2,...,k\},\sum\limits_{i\in I}\frac{1}{n_{i}}\in \mathbb{Z}\bigg\}\bigg|=2^{e}.$$
For any positive integer $m$ divided by $e$ parts, we shall do this by the induction on $e$ . First of all, we consider $e=1$. At this time, the factorization of $m$ reduces to
$$m=\sum_{j=1}^{l_{1}}\frac{1}{p_{j}}+\frac{\tilde{t_{1}}}{\tilde{s_{1}}}+\frac{1}{\tilde{s_{1}}\prod_{j=1}^{l_{1}}{p_{j}}}$$
This is equivalent to showing that $\sum_{j\in J}\frac{1}{p_{j}}+\frac{q}{\tilde{s_{1}}}\not\in \mathbb{Z}$ for any $J\subset\{1,2,...,l_{1}\}$ and $q\leq \tilde{t_{1}}$.
The case $J=\emptyset$ is trivial. Provided that $J\neq\emptyset$, since $(\tilde{s_{1}},\prod\limits_{j\in J}{p_{j}})$=1, we can easily obtain that
$$\tilde{s_{1}}(\sum_{j\in J}\frac{1}{p_{j}}+\frac{q}{\tilde{s_{1}}})=\tilde{s_{1}}(\sum_{j\in J}\frac{1}{p_{j}})+q\not\equiv 0 \pmod {\tilde{s_{1}}}.$$
Hence $\sum\limits_{j\in J}\frac{1}{p_{j}}+\frac{q}{\tilde{s_{1}}}$ is not an integer.

Now suppose that we have showing that for $e-1$, we consider the case of $e$. If
$$\sum_{c=1}^{e}\bigg\{\sum_{j=l_{c}+1}^{l_{c+1}}\frac{\delta_{c,j}}{p_{j}}+\frac{\tilde{q_{c}}}{\tilde{s_{c}}}+
\frac{\delta_{c}}{\tilde{s_{c}}\prod_{j=l_{c}+1}^{l_{c+1}}{p_{j}}}\bigg\}\in\mathbb{Z},$$
where $\delta_{c,j}=0$ or $1$, $\delta_{c}=0$ or $1$ and $\tilde{q_{c}}\leq\tilde{t_{c}}$ for any $1\leq c\leq e$. Let
$$A=\sum_{c=1}^{e-1}\bigg\{\sum_{j=l_{c}+1}^{l_{c+1}}\frac{\delta_{c,j}}{p_{j}}+\frac{\tilde{q_{c}}}{\tilde{s_{c}}}+
\frac{\delta_{c}}{\tilde{s_{c}}\prod_{j=l_{c}+1}^{l_{c+1}}{p_{j}}}\bigg\},
B=\sum_{j=l_{e}+1}^{l_{e+1}}\frac{\delta_{e,j}}{p_{j}}+\frac{\tilde{q_{e}}}{\tilde{s_{e}}}+
\frac{\delta_{e}}{\tilde{s_{e}}\prod_{j=l_{e}+1}^{l_{e+1}}{p_{j}}}.$$
Note that both
$(\prod_{c=1}^{e-1}\tilde{s_{c}}\prod_{j=l_{c}+1}^{l_{c+1}}{p_{j}})\cdot A$ and $(\prod_{c=1}^{e-1}\tilde{s_{c}}\prod_{j=l_{c}+1}^{l_{c+1}}{p_{j}})\cdot(A+B)$ lie in the $\Z$, so that
$(\prod_{c=1}^{e-1}\tilde{s_{c}}\prod_{j=l_{c}+1}^{l_{c+1}}{p_{j}})\cdot B$ belongs to $\Z$, which implies that $B\in \Z$,
consequently $A$ is also a integer.

By the assumption of the induction, one has $2^{e-1},2$ choices to get an integer from partial sums of $A$ and $B$ independently. Hence for any $m \in \Z^{+}$, we conclude that
$$\bigg|\bigg\{I:I\subset \{1,2,...,k\},\sum\limits_{i\in I}\frac{1}{n_{i}}\in \mathbb{Z}\bigg\}\bigg|=2^{e}.$$
\end{proof}

\begin{acknowledgment} We thank Professor Zhi-Wei Sun and Hao Pan for their helpful comments.
\end{acknowledgment}

     \end{document}